\documentclass{amsart}
\usepackage{amsfonts}
\usepackage{amsmath}
\usepackage{amssymb}
\usepackage{amsthm}

\newtheorem{theorem}{Theorem}
\newtheorem{prop}{Proposition}
\newtheorem{lemma}{Lemma}
\newtheorem{rem}{Remark}
\newtheorem{exmp}{Example}
\newtheorem{cor}{Corollary}

\begin{document}
\author{Mark Pankov}
\title{On distance two in Cayley graphs of Coxeter groups}
\address{Department of Mathematics and Computer Science, University of Warmia and Mazury,
S{\l}oneczna 54, Olsztyn, Poland}
\email{pankov@matman.uwm.edu.pl, markpankovmath@gmail.com}

\maketitle

\begin{abstract}
We consider the Cayley graph ${\rm C}(W,S)$ of a Coxeter system $(W,S)$
and describe all maximal $2$-cliques in this graph, 
i.e. maximal subsets in the vertex set such that  the distance between any two distinct elements is equal to $2$.
As an application, we show that every automorphism of the half of Cayley graph is uniquely extendable 
to an automorphism of the Cayley graph if $|S|\ge 5$.
\end{abstract}

\section{Introduction}
The distance between two vertices in a connected graph is defined as
the smallest number $m$ such that the vertices are connected by a path consisting of $m$ edges.
Two vertices are said to be $2$-{\it adjacent} if the distance between them is equal to $2$.
A clique in a graph is a subset of the vertex set, where any two distinct vertices are adjacent
(connected by an edge).
We say that a subset in the vertex set is a $2$-{\it clique} if any two distinct elements of
this subset are $2$-adjacent vertices.

The main object of this paper is the Cayley graph ${\rm C}(W,S)$ of a Coxeter system $(W,S)$.

If $S$ consists of $n$ mutually commuting involutions then 
the Cayley graph is the $n$-dimensional hypercube graph.
It is well-known that this graph contains precisely two types of maximal $2$-cliques
(maximal cliques of the corresponding half-cube graph).
Every maximal $2$-clique of the first type consists of the $n$ vertices adjacent to a given vertex
and maximal $2$-cliques of the second type are formed by four elements.

We show that the Cayley graph of any Coxeter system contains 
at most three types of maximal $2$-cliques (Theorem \ref{theorem1}).
The first type is $Sw$, $w\in W$ (our Cayley graph is left).
Maximal $2$-cliques of the second type correspond to triples of mutually non-adjacent nodes in 
the associated  Coxeter diagram and contain precisely four elements.
The third type is related to unlabeled edges of the Coxeter diagram, 
i.e. pairs $s,s'\in S$ such that the order of $ss'$ is $3$.
Every maximal $2$-clique of this type consists of three elements.

In Section 5, we consider the half of Cayley graph (a generalization of the half-cube graph).
Using Theorem \ref{theorem1}, we show that every automorphism of this graph  
can be uniquely extended to an automorphism of the Cayley graph if $|S|\ge 5$.

\section{Coxeter systems and associated Cayley graphs}
Let $W$ be a group generated by a finite set $S$.
Suppose that every element of $S$ is an involution.
For distinct $s,s'\in S$ we denote by $m(s,s')$ the order of the element $ss'$.
Then $m(s,s')=m(s',s)$ and
the condition $m(s,s')=2$ is equivalent to the commuting of $s$ and $s'$.

From this moment we suppose that $(W,S)$ is a {\it Coxeter system}. 
This means that 
the group $W$ has the following presentation
$$W=\langle\,S: (ss')^{m(s,s')}=1, \;(s,s')\in {\mathfrak I}\,\rangle,$$
where ${\mathfrak I}$ is the set of all pairs $(s,s')$ such that $m(s,s')$ is finite.
Our Coxeter system is completely defined by the associated 
{\it Coxeter diagram} whose nodes are the elements of $S$.
The nodes corresponding  to $s$ and $s'$ are connected by an edge only in the case when $m(s,s')\ge 3$
(the involutions $s$ and $s'$ are non-commuting). 
If $m(s,s')\ge 4$ then the edge connecting $s$ and $s'$
is labeled by the number $m(s,s')$.
All spherical and affine Coxeter systems are known,
see \cite[Appendix 1]{BB} or \cite[Table 6.1]{Davis} for the corresponding irreducible diagrams.

The {\it Cayley graph} ${\rm C}(W,S)$ is the graph whose vertex set is $W$
and $w,v\in W$ are adjacent vertices if $v=sw$ for a certain $s\in S$
(since $S$ consists of involutions, the adjacency relation is symmetric).
In contrast to \cite[Section 2.1]{Davis}, our Cayley graph is left.
In the right Cayley graph, $w,v\in W$ are adjacent if $v=ws$ for a certain $s\in S$.
The mapping $w\to w^{-1}$ is an isomorphism between these graphs.

The Cayley graph of the dihedral group $\textsf{I}_{2}(n)$ is the $(2n)$-cycle.
See \cite[Figures 3.2 and 3.3]{BB} for the Cayley graphs of $\textsf{A}_{3}$ and $\textsf{H}_{3}$.

For every $w\in W$ the right multiplication 
$$R_{w}: v\to vw$$
is an automorphism of the Cayley graph.
If the diagram of our Coxeter system does not contain adjacent edges labeled by $\infty$
then the automorphism group of the Cayley graph is the semidirect product of $W$
and the automorphism group of the diagram \cite[Corollary 3.2.6]{BB}.

\begin{rem}{\rm
The Cayley graph can be identified with the graph whose vertices 
are maximal simplices of the Coxeter complex $\Sigma(W,S)$ and 
two maximal simplices are adjacent vertices of the graph if their intersection consists of $|S|-1$ elements
\cite{Tits}.
}\end{rem}

The {\it length} $l(w)$ of an element $w\in W$
is the smallest number $m$ such that $w$ has an expression 
\begin{equation}\label{eq-express}
w=s_{1}\dots s_{m},\;\;s_{1},\dots,s_{m}\in S.
\end{equation}
This is the distance between $1$ and $w$ in the Cayley graph.
The distance $d(w,v)$ between $w,v\in W$ is equal to 
$$l(wv^{-1})=l(vw^{-1}).$$
In what follows we will say that  \eqref{eq-express} is a {\it reduced}  expression if $m=l(w)$.

If $u$ and $u'$ are adjacent to $w$ then 
$u=sw$ and $u'=s'w$ for some $s,s'\in S$. Since $d(s,s')=2$,
the elements $u$ and $u'$ are not adjacent. 
Therefore, every maximal clique of the Cayley graph is a pair of adjacent vertices.

We will use the following well-known properties of Coxeter systems, 
see \cite{BB,Davis}.

\begin{theorem}[The exchange condition]\label{th-exch}
For every reduced expression \eqref{eq-express}
and every $s\in S$ satisfying $l(sw)\le m$ there exists 
$k\in \{1,\dots,m\}$ such that 
$$sw=s_{1}\dots \hat{s}_{k}\dots s_{m}$$
{\rm(the symbol $\hat{}$ means that the corresponding term is omitted)}.
\end{theorem}

\begin{prop}\label{prop-b1}
For every $w\in W$ there is a subset $S_{w}\subset S$
such that every reduced expression of $w$  is formed by all elements of $S_{w}$.
\end{prop}

\begin{prop}\label{prop-b2}
The group $W$ cannot be spanned by a proper subset of $S$. 
\end{prop}

\section{Maximal $2$-cliques}
First we present three examples of $2$-cliques in the Cayley graph ${\rm C}(W,S)$.

\begin{exmp}[First type]\label{exmp1}{\rm
 Any two elements of $S$ are $2$-adjacent and $S$ is a $2$-clique.
Since the right multiplication $R_{w}$ is an automorphism of the Cayley graph,
$Sw$ is a $2$-clique for every $w\in W$. 
}\end{exmp}

\begin{rem}\label{rem1}{\rm
Suppose that $S=Sw$. Then for any $s_{1},s_{2}\in S$ there exist $s'_{1},s'_{2}\in S$
such that  $s_{1}=s'_{1}w$ and $s_{2}=s'_{2}w$.
If $w\ne 1$ then $s_{1}\ne s'_{1}$ and $s_{2}\ne s'_{2}$.
We have 
$$s'_{1}s_{1}=w=s'_{2}s_{2}\;\mbox{ and }\;s'_{2}s'_{1}s_{1}=s_{2}$$
By Proposition \ref{prop-b2}, the latter means that $s_{2}=s'_{1}$.
Therefore, $S=\{s_{1},s_{2}\}$ and $s_{1}s_{2}=s_{2}s_{1}$.
So, the equality $Sw=Sw'$ implies that $w=w'$ except the case when 
our Coxeter system is $\textsf{I}_{2}(2)$.
}\end{rem}

\begin{exmp}[Second type]\label{exmp3}{\rm
Let $s,s',s''$ be three mutually commuting elements of $S$.
Then $ss's''$ is $2$-adjacent to $s,s',s''$ and 
$$\{sw,s'w,s''w, ss's''w\}$$
is a $2$-clique for every $w\in W$.
}\end{exmp}

\begin{exmp}[Third type]\label{exmp2}{\rm
Suppose that $s,s'\in S$ and $m(s,s')=3$.
Then 
$$ss's=s'ss'.$$
We denote the latter element by $w(s,s')$.
It is $2$-adjacent to $s,s'$ and for every $w\in W$ the set 
$$\{sw,s'w, w(s,s')w\}$$
is a $2$-clique.
}\end{exmp}

Our main result is the following.

\begin{theorem}\label{theorem1}
Every maximal $2$-clique  of the Cayley graph ${\rm C}(W,S)$ 
is one of the $2$-cliques described above.
\end{theorem}

\begin{rem}{\rm
Every maximal $2$-clique in the Cayley graph of ${\textsf A}_{1}$
is a $2$-clique of the third type. This Cayley graph contains $2$-cliques of the first type,
but they are not maximal.
}\end{rem}

\begin{rem}{\rm
If $(W,S)$ is an irreducible spherical Coxeter system 
($W$ is a finite group and the associated Coxeter diagram is connected) and $|S|=4$
then ${\rm C}(W,S)$ contains $2$-cliques of the second type only in the case when
our Coxeter system is ${\textsf D}_{4}$.
}\end{rem}

\begin{rem}{\rm
If $S$ consists of $n$ mutually commuting involutions
then the Cayley graph is the $n$-dimensional hypercube graph.
It contains maximal $2$-cliques of the first and second types if $n\ge 4$.
In the case when $n=3$, there precisely two maximal $2$-cliques of the second type
and $2$-cliques of the first type are not maximal.
}\end{rem}

\section{Proof of Theorem \ref{theorem1}}

\begin{lemma}\label{lemma1}
If $u\in W\setminus S$ is  $2$-adjacent to three distinct $s,s',s''\in S$ 
then $s,s',s''$ are mutually commuting and $u=ss's''$.
\end{lemma}

\begin{proof}
Since $u$ is $2$-adjacent to $s,s',s''$ and $u\not\in S$, 
there are three reduced expressions
$$u=s_{1}s_{2}s,\;\;u=s'_{1}s'_{2}s',\;\;u=s'_{1}s'_{2}s'',$$
where $s_{1},s_{2},s'_{1},s'_{2},s''_{1},s''_{2}\in S$.
Proposition \ref{prop-b1} guarantees that
$$S_{u}=\{s,s',s''\}$$
and 
$$\{s_{1},s_{2}\}=\{s',s''\},\; \{s'_{1},s'_{2}\}=\{s,s''\},\;
\{s''_{1},s''_{2}\}=\{s,s'\}.$$
Thus we have the following possibilities for the first and second expressions:
\begin{enumerate}
\item[(1)] $u=s''s's=s''ss'$,
\item[(2)] $u=s''s's=ss''s'$,
\item[(3)] $u=s's''s=s''ss'$,
\item[(4)] $u=s's''s=ss''s'$.
\end{enumerate}
Case (1). The involutions $s,s'$ are commuting and the third expression is  
\begin{equation}\label{eq1}
u=ss's''=s'ss''.
\end{equation}
Then $s''s's=u=s'ss''$ and $s's''s's=ss''$.
We apply the exchange condition to $w=s''s's$ and get the following three possibilities:
\begin{enumerate}
\item[$\bullet$] $s's=ss''$,
\item[$\bullet$] $s''s=ss''$,
\item[$\bullet$] $s''s'=ss''$.
\end{enumerate}
The first and the third contradict to Proposition \ref{prop-b2}.
Thus $s$ and $s''$ are commuting.
Similarly, the equality $s''s's=u=ss's''$ shows that $ss''s's=s's''$.
Using the above arguments, we establish that $s'$ and $s''$ are commuting.

Case (2). 
The equality $s''s's=ss''s'$ implies that $s's=s''ss''s'$.
As in the previous case, we show that $s$ and $s'$ are commuting.
Then the third expression is \eqref{eq1}
which implies that $ss''s'=u=ss's''$ and $s',s''$ are commuting.
The equality 
$$s''ss'=s''s's=u=ss''s'$$  
guarantees that $s$ and $s''$ are commuting.
 
Case (3).
We have $s's''s=s''ss'$ and $s''s's''s=ss'$.
As above, this means that $s,s'$ are commuting and
the third expression is \eqref{eq1}.
Then $s's''s=u=s'ss''$ and $s,s''$ are commuting.
The equality 
$$s's''s=u=s''ss'=s''s's$$
shows that $s'$ and $s''$ are commuting.

Case (4).
Since $s's''s=ss''s'$, we have 
$$s''s=s'ss''s'\;\mbox{ and }\;ss's''s=s''s'.$$
By the above arguments, 
this guarantees that $s''$ is commuting with both $s$ and $s'$.
Then 
$$s''ss'=ss''s'=u=s's''s=s''s's$$
which implies that $s$ and $s'$ are commuting.
\end{proof}

\begin{rem}\label{rem-proof}{\rm
Later we will need the following facts established in the proof of Lemma \ref{lemma1}.
If $s,s',s''$ are distinct elements of $S$ then each of the equalities 
$$s''s's=ss''s'\;\mbox{ and }\;s's''s=s''ss'$$
implies that $s$ and $s'$ are commuting, see the cases (2) and (3).
The equality  $$s's''s=ss''s'$$
guarantees that $s,s',s''$  are mutually commuting, see the case (4).
}\end{rem}

Lemma \ref{lemma1} shows that for any three mutually commuting 
$s,s',s''\in S$ the $2$-clique formed by $s,s',s''$ and $ss's''$ is maximal.
Therefore, every $2$-clique of the second type is maximal.

\begin{lemma}\label{lemma2}
If $u\in W\setminus S$ is  $2$-adjacent to $s,s'\in S$ 
then one of the following possibilities is realized:
\begin{enumerate}
\item[$\bullet$] $m(s,s')=3$ and $u=ss's=s'ss'$,
\item[$\bullet$]  $s,s'$ are commuting and $u=s''s's$
for a certain $s''\in S$.
\end{enumerate}
\end{lemma}

\begin{proof}
Since $u$ is $2$-adjacent to $s,s'$ and $u\not\in S$, 
there are two reduced expressions
$$u=s_{1}s_{2}s\;\mbox{ and }\;u=s'_{1}s'_{2}s',$$
where $s_{1},s_{2},s'_{1},s'_{2}\in S$.
By Proposition \ref{prop-b1}, 
$$\{s,s_{1},s_{2}\}=S_{u}=\{s',s'_{1},s'_{2}\}.$$
If $|S_{u}|=2$ then $S_{u}=\{s,s'\}$ and 
$$u=ss's=s'ss'$$
which implies that $m(s,s')=3$, i.e. the first possibility is realized.

If $|S_{u}|=3$ then $S_{u}=\{s,s',s''\}$ 
and, as in the proof of Lemma \ref{lemma1},
we have the following possibilities for the above expressions:
\begin{enumerate}
\item[(1)] $u=s''s's=s''ss'$,
\item[(2)] $u=s''s's=ss''s'$,
\item[(3)] $u=s's''s=s''ss'$,
\item[(4)] $u=s's''s=ss''s'$.
\end{enumerate}
Each of these equalities gives the second possibility. 
Indeed, the case (1) is trivial
and for the cases (2) -- (4) this follows from Remark \ref{rem-proof}.
\end{proof}

By Lemma \ref{lemma2}, 
for any $s,s'\in S$ satisfying $m(s,s')=3$
the $2$-clique formed by $s,s'$ and $ss's=s'ss'$ is maximal.
Thus every $2$-clique of the third type is maximal. 

Now we prove Theorem \ref{theorem1}.
Let $C$ be a maximal $2$-clique of the Cayley graph. 
For any $u,u'\in C$ there exist $w\in W$ and $s,s'\in S$ such that 
$u=sw$ and $u'=s'w$.
The maximal $2$-clique $Cw^{-1}$ contains $s$ and $s'$.
Thus we can suppose that $C$ contains at least two distinct elements of $S$.

So, let $s$ and $s'$ be elements of $S$ belonging to $C$.
Suppose that $C\ne S$, i.e. there is $u\in C\setminus S$.
We have the following possibilities:
\begin{enumerate}
\item[$\bullet$] there is a third element $s''\in S$ contained in $C$,
\item[$\bullet$] $C$ contains precisely two elements of $S$.
\end{enumerate}
In the first case, Lemma \ref{lemma1} implies that
$s,s',s''$ are mutually commuting and 
$$C=\{s,s',s'', u=ss's''\}.$$
In the second case, Lemma \ref{lemma2} gives the following possibilities: 
$$m(s,s')=3\;\mbox{ and }\;C=\{s,s',ss's=s'ss'\}$$
or $s,s'$ are commuting  and $u=s''s's$
for a certain $s''\in S$. 
The latter means that the maximal $2$-clique $Cs's$ 
contains $s,s',s''$, i.e. 
$$Cs's=S\;\mbox{ or }\; Cs's=\{s,s',s'',u\}.$$
Then $C$ is a $2$-clique of the first type or the second type.

\section{The half of Cayley graph}
The group $W$ can be presented as the disjoint union of the following subsets
$$W_{1}:=\{\;w\in W:l(w)\mbox{ is odd }\}\;\mbox{ and }\;
W_{2}:=\{\;w\in W:l(w)\mbox{ is even }\}.$$
Using the exchange condition we establish the following:
\begin{enumerate}
\item[$\bullet$] the distance between any two elements of $W_{i}$, $i\in \{1,2\}$ is even,
\item[$\bullet$] the distance between every element of $W_{1}$ and every element of $W_{2}$ is odd,
\end{enumerate}
Also, note that $W_{2}$ is a subgroup of $W$.

Consider the graph $\Gamma_{i}$, $i\in \{1,2\}$ whose vertex set is $W_{i}$
and two elements of $W_{i}$ are adjacent vertices 
if the distance between them (in the Cayley graph) is equal to $2$.
The right multiplication $R_{w}$ preserves both $W_{i}$ if $w\in W_{2}$.
If $w\in W_{1}$ then $R_{w}$ transfers $W_{1}$ to $W_{2}$ and conversely,
i.e. $R_{w}$ induces an isomorphism between $\Gamma_{1}$ and $\Gamma_{2}$. 

\begin{rem}{\rm
In the case when $S$ consists of mutually commuting involutions,
we get the well-known {\it half-cube graph}.
}\end{rem}

Every $2$-clique of the Cayley graph is contained in $W_{1}$ or $W_{2}$
and every maximal clique of $\Gamma_{i}$, $i\in \{1,2\}$ is a maximal $2$-clique of ${\rm C}(W,S)$
contained in $W_{i}$.

\begin{cor}\label{cor1}
If $|S|\ge 5$ then every isomorphism between $\Gamma_{i}$ and $\Gamma_{j}$ $i,j\in \{1,2\}$  
can be uniquely extended to an automorphism of the Cayley graph.
\end{cor}

\begin{proof}
We consider the case when $i=j=1$.
Let $f:W_{1}\to W_{1}$ be an automorphism of $\Gamma_{1}$. 
Then $f$ preserves the family of maximal cliques of $\Gamma_{1}$.
Every maximal clique of $\Gamma_{1}$ is a maximal $2$-clique of ${\rm C}(W,S)$
contained in $W_{1}$. 
By Theorem \ref{theorem1}, there are precisely three types of such subsets.
They contain $|S|$ vertices, $4$ vertices and $3$ vertices, respectively.
The condition $|S|\ge 5$ guarantees that $f$ preserves the types of maximal cliques.

If $w\in W_{2}$ then $Sw$ is a maximal clique of $\Gamma_{1}$
and $f(Sw)=Sw'$ for a certain $w'\in W_{2}$.
We set $f(w):=w'$ and get a bijective transformation of $W$.

If $w,v\in W$ are adjacent vertices of the Cayley graph
then one of these vertices belongs to $W_{1}$ and the other is an element of $W_{2}$.
Suppose that $v\in W_{1}$ and $w\in W_{2}$. 
Then $v\in Sw$ and $f(v)\in f(Sw)=Sf(w)$ which implies that $f(v)$ and $f(w)$
are adjacent vertices of the Cayley graph.
The apply the same arguments to $f^{-1}$ and establish that $f$ is an automorphism of the Cayley graph.

The uniqueness of such extension follows from the fact that  
$w$ is the unique vertex of the Cayley graph adjacent to all vertices from $Sw$ (Remark \ref{rem1}). 
\end{proof}

\begin{rem}{\rm
In the case when $|S|=4$, the latter statement fails.
If $S$ consists of $4$ mutually commuting involutions then 
the Cayley graph is the $4$-dimensional hypercube graph.
There are automorphisms of the associated half-cube graph which transfer 
$2$-cliques of the first type to $2$-cliques of the second type
and conversely. 
They are not extendable to automorphisms of the hypercube. 
}\end{rem}

\begin{cor}\label{cor2}
Suppose that $|S|\ge 5$.
Let $f$  be a bijective transformation of $W$ preserving  the distance $2$ in both directions, i.e.
$$d(w,v)=2\;\Longleftrightarrow\;d(f(w),f(v))=2$$
for all $w,v\in W$.
Then there are two automorphisms $f_{1}$ and $f_{2}$ of the Cayley graph
such that the restriction of $f$ to $W_{i}$, $i\in \{1,2\}$ coincides with 
the restriction of $f_{i}$ to $W_{i}$
\footnote{The transformation $f$ is an automorphism of the Cayley graph if and only if $f_{1}=f_{2}$.}.
\end{cor}

\begin{proof}
It is not difficult to show that $f$ preserves both $W_{i}$ or 
transfers $W_{1}$ to $W_{2}$ and conversely.
Corollary \ref{cor1} gives the claim.
\end{proof}

\subsection*{Acknowledgment}
The author thanks John D. Dixon for the interesting and remarks.

\end{document}